\documentclass{amsart}

\usepackage{amssymb}
\usepackage{color}
\usepackage{enumerate}
\usepackage{hyperref}

\theoremstyle{plain}
\newtheorem{theorem}{Theorem}[section]

\newtheorem{proposition}[theorem]{Proposition}
\newtheorem{corollary}[theorem]{Corollary}

\theoremstyle{definition}

\newtheorem{conjecture}[theorem]{Conjecture}

\theoremstyle{remark}
\newtheorem{remark}[theorem]{Remark}

\newtheorem{case}{Case}

\newcommand{\N}{\mathbb{N}}
\renewcommand{\O}{\mathcal{O}}
\renewcommand{\P}{\mathbb{P}}
\newcommand{\Q}{\mathbb{Q}}

\newcommand{\Z}{\mathbb{Z}}
\newcommand{\Qq}{\Q_\mathrm{quad}}

\renewcommand{\l}{\lambda}

\renewcommand{\t}{\tau}

\newcommand{\es}{\emptyset}
\newcommand{\ol}{\overline}

\newcommand{\gal}{\mathrm{Gal}}

\newcommand{\set}[1]{\left\{#1\right\}}
\newcommand{\tup}[1]{\left<#1\right>}

\begin{document}

\title{On quadratic periodic points of quadratic polynomials}

\author{Zhiming Wang}
\address[Zhiming Wang]{Department of Mathematics, Stanford University}
\email{zmwang@stanford.edu}

\author{Robin Zhang}
\address[Robin Zhang]{Department of Mathematics, Stanford University}
\email{robinz16@stanford.edu}

\date{\today}

\maketitle

\begin{abstract}
  Bounding the number of preperiodic points of quadratic polynomials
  with rational coefficients is one case of the Uniform Boundedness
  Conjecture in arithmetic dynamics. Here, we provide a general
  framework that may reduce finding periodic points of such
  polynomials over Galois extensions of $\Q$ to finding periodic
  points over the rationals. Furthermore, we present evidence that
	there are no such polynomials (up to linear conjugation)
  with periodic points of exact period 5 in quadratic fields by
  searching for points on an algebraic curve that classifies quadratic
	periodic points of exact period 5 and suggesting the application of the method
  of Chabauty and Coleman for further	progress.
\end{abstract}

\section{Introduction}
\label{sec:intro}

The principal goal of the study of a discrete dynamical system is to
classify points of a set according to their orbits under a
self-map. Finite orbits, i.e. orbits of period points and preperiodic
points\footnote{%
  For a discrete dynamical system consisting of a set $S$ and a
  self-map $\phi: S \to S$, a point $\alpha \in S$ is called a
  \emph{preperiodic point} if $\phi^{m+n}(\alpha) = \phi^m(\alpha)$
  for some $m \ge 0$ and $n \ge 1$.}
, are of particular interest for obvious reasons. In the field of
arithmetic dynamics, it is natural to further impose number theoretic
conditions on the periodic or preperiodic points, e.g., that the
points be rational, or be in a certain number field. One profound
problem in this field---the Uniform Boundedness Conjecture---follows
this line of thought and generalizes Merel's Theorem on torsion points
of elliptic curves. The Uniform Boundedness Conjecture in arithmetic
dynamics posits that the number of preperiodic points of a rational
morphism over a number field is not only finite (Northcott, 1950
\cite{MR0034607}), but can be bounded by a number depending only on a
few general parameters. The precise statement is as follows.

\begin{conjecture}[Morton-Silverman, 1994~\cite{MR1264933}]
  Fix integers $d \ge 2$, $n \ge 1$, and $D \ge 1$. There is a
  constant $C(d, n, D)$ such that for all number fields $K/\Q$ of
  degree at most $D$ and all morphisms $\phi: \P^n \to \P^n$ of degree
  $d$ defined over $K$, the number of preperiodic points of $\phi$
  over $\P^n(K)$ is bounded above by $C(d, n, D)$.
\end{conjecture}

Little progress has been made on this conjecture. In fact, even the
simplest case $(d, n, D) = (2, 1, 1)$, i.e., the problem of bounding
the number of rational preperiodic points of quadratic rational
morphisms, still awaits treatment. In this paper, we will loosely
stick to the case $(d, n, D) = (2, 1, 2)$, and further specialize to
quadratic polynomials with rational coefficients. Note that linear
conjugation on $\phi$ does not affect the size of orbits, so all such
polynomials can be put into the standard form $\phi(z) = z^2 + c$,
where $c$ is rational.

After the completion of this project it was brought to our attention
that Hutz and Ingram \cite{MR3065461} investigated the $(d, n, D) =
(2, 1, 2)$ case and provide further strong computational evidence for
\ref{cj:n=5-zero} from the different perspective of dealing
with all N for various rational values of c as opposed to our approach
of various values of N for all rational c. Doyle, Faber, and
Krumm \cite{1309.6401} also produced related results in this case from
considerations of the finite directed graphs of $K$-rational preperiodic
points for quadratic number fields $K$. Also, Morton and Patel \cite{MortonPatel}
provide different considerations of the general connections between Galois theory
and periodic points which are related to our results in Section~\ref{sec:galois}.
%!! Refine mentions

Before presenting our main results, we introduce some notations and
definitions.

\subsection{Notations and definitions}
\label{subsec:def}

Let $\phi(z) = z^2 + c$ where $c$ is rational (from this point onward,
all polynomial coefficients are rational unless otherwise
specified). Let $\phi^N$ denote the $N$-th iteration of $\phi$. If a
number $z$ in some number field satisfies $\phi^N(z) = z$, then $z$ is
called a \emph{periodic point} of $\phi$ of \emph{period $N$}, and its
orbit
\[
(z,\, \phi(z),\, \phi^2(z),\, \dots,\, \phi^{N-1}(z))
\]
is called an \emph{$N$-cycle} of $\phi$. The \emph{trace} of the
$N$-cycle is defined to be the sum of all its elements. Furthermore,
if $\phi(z),\, \dots, \phi^{N-1}(z)$ are all distinct from $z$, then
the orbit is called an \emph{exact $N$-cycle}, and $z$ is called a
periodic point of \emph{exact period $N$}. For convenience, we will
frequently refer to periodic points of exact periodic $N$ as
\emph{$N$-periodic points}.

In this paper, instead of considering periodic points in any single
quadratic number field, we will consider the total number of periodic
points in all quadratic number fields. For this purpose, we introduce
the notation $\Qq$ for the union of all quadratic number
fields. Formally,
\[
\Qq := \set{\alpha \in \ol{\Q}: \text{$p(\alpha) = 0$ for some $p(x)
\in \Q[x]$ of degree 2}}.
\]
With this notation, we can now state our main results.

\subsection{Main results}
\label{subsec:results}

We would like to classify all $N$-periodic points of quadratic
polynomials $\phi_c(z) = z^2 + c$ in $\Qq$. Note however that for any
given $c$, all $N$-periodic points of $\phi_c$ must be roots of the
polynomial $\phi_c^N(z) - z = 0$ in $z$, so it suffices to classify
all rational values of $c$ such that $\phi_c$ has an $N$-periodic
point in $\Qq$. This will be the subject of this paper.

The following theorem and the conjecture derived from it reveal
certain connections between different elements of an $N$-cycle, and in
some cases impose strong restrictions on the trace of the cycle that
help pin down the $N$-periodic points. Galois action will permute
periodic points because they are roots of a dynatomic polynomial with
rational coefficients, but it is not obvious whether points in the same
cycle will be Galois conjugates.

\begin{theorem}
  \label{thm:galois}
  Let $N \in \N^*$, $c \in \Q$, $\phi_c(z) = z^2 + c$, and $K$ be a
  Galois extension of $\Q$ with degree $d = [K : \Q]$. Let
  $(z_0, \dots, z_{N-1})$ be an exact $N$-cycle of $\phi_c$, where
  $z_j \in K$ and $\phi_c(z_j) = z_{j+1}$ for all $j \in \Z/N\Z$. Let
  $g = \gcd(N, d)$. Then exactly one of the following holds:
  \begin{enumerate}[(i)]
  \item $z_{m \cdot \frac{N}{g}} = \t(z_0)$ for some $m \in \Z/g\Z$
    and some nontrivial $\t \in \gal(K/\Q)$;

  \item
    $\set{z_0, \dots, z_{N-1}} \cap \set{\t(z_0), \dots, \t(z_{N-1})}
    = \es$ for all nontrivial $\t \in \gal(K/\Q)$.
  \end{enumerate}
\end{theorem}

We have good reasons to believe that the second case never occurs,
hence we propose the following conjecture.

\begin{conjecture}
  \label{cj:galois}
  Let $N \in \N^*$, $c \in \Q$, $\phi_c(z) = z^2 + c$, and $K$ be a
  Galois extension of $\Q$ with degree $d = [K : \Q]$. Let
  $(z_0, \dots, z_{N-1})$ be an exact $N$-cycle of $\phi_c$, where
  $z_j \in K$ and $\phi_c(z_j) = z_{j+1}$ for all $j \in \Z/N\Z$. Let
  $g = \gcd(N, d)$. Then there is some $m \in \Z/g\Z$ and some
  nontrivial $\t \in \gal(K/\Q)$ such that
  $z_{m \cdot \frac{N}{g}} = \t(z_0)$.
\end{conjecture}

In the $d = 2$ case, i.e., the quadratic case that we are mostly
concerned about in this paper, Conjecture~\ref{cj:galois} reduces
finding quadratic $N$-periodic points to finding rational points on
certain algebraic curves, as we will see in
Section~\ref{sec:implications}. In particular, in the period 5 case,
it implies the following conjecture.

\begin{conjecture}
  \label{cj:n=5-zero}
  There are no rational values $c$ such that $\phi_c(z) = z^2 + c$ has
  a 5-periodic point in $\Qq$.
\end{conjecture}

Conjecture~\ref{cj:n=5-zero} is also supported by the following
well-known consequence of Faltings's Theorem since due to the high
genus of corresponding algebraic curves $C_1(5)$ and $C_0(5)$ as
described by Flynn, Poonen, and Schaefer \cite{MR1480542}. Furthermore,
recent work by Hutz and Ingram \cite{MR3065461} and Doyle, Faber,
and Krumm \cite{1309.6401} also provide discussions of
Conjecture~\ref{cj:n=5-zero}.

\begin{proposition}
  \label{prop:n=5-finite}
  There are finitely many rational values $c$ such that
  $\phi_c(z) = z^2 + c$ has a 5-periodic point in $\Qq$.
\end{proposition}

In Section~\ref{sec:n=5}, we study a genus 11 curve with rational points
corresponding to rational $c$ and corresponding 5-cycles in $\Qq$. We present
computational evidence towards Conjecture~\ref{cj:n=5-zero} in a search
for rational points on a corresponding algebraic curve. Perhaps through a
clever application of Coleman and Chabauty's methods, one may prove that
the points that we have found are all such points and thus obtain a
proof of Conjecture~\ref{cj:n=5-zero}.

Throughout this paper, we will perform the necessary calculations in
Mathematica (version 9.0.1.0) or Sage (version 6.2). All computational
programs can be found in our source code repository \cite{src}. The
repository also contains additional programs that provide
computational evidence for our conjectures, including one C++ program
using the FLINT library \cite{Hart2010}.

\section{Periodic points in Galois number fields}
\label{sec:galois}

The main objects of study in this paper are the periodic points of
quadratic polynomials in quadratic extensions of $\Q$. It is easy to
see that these quadratic extensions are automatically Galois over
$\Q$.  The property of being Galois alone leads to an interesting
result for Galois extensions of general degrees, following the
observation that polynomial maps (with rational coefficients) commute
with Galois conjugations.

\begin{theorem}[Restatement of Theorem~\ref{thm:galois}]
  \label{thm:galois-re}
  Let $N \in \N^*$, $c \in \Q$, $\phi_c(z) = z^2 + c$, and $K$ be a
  Galois extension of $\Q$ with degree $d = [K : \Q]$. Let
  $(z_0, \dots, z_{N-1})$ be an exact $N$-cycle of $\phi_c$, where
  $z_j \in K$ and $\phi_c(z_j) = z_{j+1}$ for all $j \in \Z/N\Z$. Let
  $g = \gcd(N, d)$. Then exactly one of the following holds:
  \begin{enumerate}[(i)]
  \item $z_{m \cdot \frac{N}{g}} = \t(z_0)$ for some $m \in \Z/g\Z$
    and some nontrivial $\t \in \gal(K/\Q)$;

  \item
    $\set{z_0, \dots, z_{N-1}} \cap \set{\t(z_0), \dots, \t(z_{N-1})}
    = \es$ for all nontrivial $\t \in \gal(K/\Q)$.
  \end{enumerate}
\end{theorem}

\begin{proof}
  First note that (i) and (ii) cannot be simultaneously true. In fact,
  if (i) is true, i.e., $z_{m \cdot \frac{N}{g}} = \t(z_0)$ for some
  $m$ and nontrivial $\t \in \gal(K/\Q)$, then
  $z_{m \cdot \frac{N}{g}} \in \set{z_0, \dots, z_{N-1}} \cap
  \set{\t(z_0), \dots, \t(z_{N-1})}$, so (ii) is false.

  If (ii) is true, then we are done. Otherwise, (ii) is false, so we
  have some nontrivial $\t \in \gal(K/\Q)$ such that $z_k = \t(z_j)$
  for some $j, k \in \Z/N\Z$. Note that $\t$ commutes with $\phi_c$
  (the polynomial $\phi_c$ is defined over $\Q$, and $\t$ is a field
  automorphism fixing the base field $\Q$), so we may assume $j = 0$;
  otherwise, without loss of generality $j \le N$, then
  $\t(z_0) = \t(z_N) = \t(\phi_c^{N-j}(z_j)) = \phi_c^{N-j}(\t(z_j)) =
  \phi_c^{N-j}(z_k) = z_{N+k-j}$,
  so we may set $j$ to 0 and $k$ to $N+k-j$. Assuming $j = 0$, we have
  $\t(z_0) = z_k$; this, together with the fact that $\t$ commutes
  with $\phi_c$, implies that $\t \equiv \phi_c^k$ on the entire cycle
  $(z_0, \dots, z_{N-1})$.

  Now recall that $K$ is Galois, so the order of the Galois group
  $\gal(K/\Q)$ is exactly $[K : \Q] = d$. Therefore, by Lagrange's
  theorem, $\t^d = id$, and hence
  \[
  z_0 = \t^d(z_0) = (\phi_c^k)^d(z_0) = z_{kd}.
  \]
  Let $r$ be the remainder of $kd$ modulo $N$. If $r$ is nonzero, then
  $(z_0, \dots, z_{r-1})$ forms a cycle of $\phi_c$ with length
  $r < N$, violating the assumption that $(z_0, \dots, z_{N-1})$ is an
  exact $N$-cycle. Therefore, $r = 0$, i.e., $N$ divides $kd$.
  Consequently, $k$ is a multiple of
  $\frac{N}{\gcd(N,d)} = \frac{N}{g}$, whence we have some
  $m \in \Z/g\Z$ such that $k = m \cdot \frac{N}{g}$. Recall that
  $\t(z_0) = z_k$, so (i) is true, and we are done.
\end{proof}

In fact, we have reasons to believe that the second case never occurs
in general. For instance, Panraksa \cite{MR2982105} proved a specific
version of our theorem for $N = 4$, $d = 2$, and showed that
$\{z_0, z_1, z_2, z_3\} \cap \{\ol{z_0}, \ol{z_1}, \ol{z_2},
\ol{z_3}\} \ne \es$,
hence rejecting the second case (for $d=2$, the only nontrivial
element of $\gal(K/\Q)$ is usual conjugation). Furthermore, all of
the examples currently known to us (the 5-cycles described by Flynn,
Poonen, and Schaefer \cite{MR1480542}, and the 6-cycles described by
Stoll \cite{MR2465796}) fall into the first case. Our computational
efforts also seem to favor this claim. Therefore, we propose the
following conjecture.

\begin{conjecture}[Restatement of Conjecture~\ref{cj:galois}]
  \label{cj:galois-re}
  Let $N \in \N^*$, $c \in \Q$, $\phi_c(z) = z^2 + c$, and $K$ be a
  Galois extension of $\Q$ with degree $d = [K : \Q]$. Let
  $(z_0, \dots, z_{N-1})$ be an exact $N$-cycle of $\phi_c$, where
  $z_j \in K$ and $\phi_c(z_j) = z_{j+1}$ for all $j \in \Z/N\Z$. Let
  $g = \gcd(N, d)$. Then there is some $m \in \Z/g\Z$ and some
  nontrivial $\t \in \gal(K/\Q)$ such that
  $z_{m \cdot \frac{N}{g}} = \t(z_0)$.
\end{conjecture}

Implications of this conjecture will be deferred to
Section~\ref{sec:implications}, after we set up the necessary
geometric model.

\section{Geometric model}
\label{sec:model}
In this section we characterize the $N$-periodic points of rational
functions by geometrically irreducible algebraic curves using a well-
known model. After re-establishing the geometric model, we may apply
machinery and previous results from algebraic geometry to study these
periodic points, which are otherwise algebraic objects. (Note that
Theorem~\ref{thm:galois-re} is a purely algebraic result.)

Let $\phi_c(z) = z^2 + c$, then all $d$-periodic points of $\phi_c$,
where $d \mid N$, satisfy the polynomial equation
\[
\phi_c^N(z) - z = 0.
\]
By the M\"obius inversion formula, we have
\[
\phi_c^N(z) - z = \prod_{d|N} \Phi_d(z, c),
\]
where the dynatomic polynomial
\[
\Phi_d(z, c) = \prod_{m|d}(\phi_c^m(z) - z)^{\mu(d/m)}.
\]
With a little bit of effort we can show that $\Phi_d(z, c)$ belongs to
$\Z[z, c]$, and that all $N$-periodic of $\phi_c$ are roots of the
polynomial equation $\Phi_N(z, c) = 0$ (but not necessarily the
converse---$\Phi_d(z, c)$ might still contain roots with exact period
smaller than $N$).

The polynomial equation $\Phi_N(z, c)$ defines an algebraic curve in
the $(z, c)$-plane. Denote the normalization of this curve by
$C_1(N)$. Observe that the map $\phi_c$ permutes $N$-cycles, so the
map $\sigma: (z, c) \mapsto (\phi_c(z), c)$ is an automorphism of the
curve $C_1(N)$, and it generates a group $\tup{\sigma}$ of order
$N$. Take the quotient curve $C_1(N)/\tup{\sigma}$, and denote the
normalization of the quotient curve by $C_0(N)$. Note that for a given
number field $K$, the $K$-points on $C_0(N)$ do not necessarily
correspond to $K$-points on $C_1(N)$; rather, they correspond to
$\gal(\ol{K}/K)$-stable orbits on $C_1(N)$.

From the above discussions, the study of periodic points of exact
period $N$ in a number field $K$ (where we also require that
$c \in K$) reduces to the study of $K$-points on the curves $C_1(N)$
and $C_0(N)$. $K$-points on $C_1(N)$ correspond directly, with
finitely many exceptions due to removal of singularities, to pairs
$(z, \phi_c)$ of a point $z \in K$ and a map $\phi_c(z) = z^2 + c$
with $c \in K$ such that $z$ is a periodic point of period $N$ (not
necessarily exact) of $\phi_c$. $K$-points on $C_0(N)$ correspond,
with finitely many exceptions, to pairs $(\O, \phi_c)$ of a
$\gal(\ol{K}/K)$-stable orbit $\O$ of size $N$ and a map
$\phi_c(z) = z^2 + c$ with $c \in K$; obviously these include all
$(\O, \phi_c)$ pairs where elements of $\O$ are strictly contained in
$K$, and hence contain full information about periodic points in $K$.

\section{Implications of the Galois conjecture}
\label{sec:implications}

In this section, we discuss the implications of
Conjecture~\ref{cj:galois-re} in the special case of $d = 2$, i.e.,
when $K$ is quadratic, in which we are most interested.  In this case,
$K$ is automatically Galois, so the conjecture can be applied
unconditionally. Also, $|\gal(K/\Q)| = 2$, where the only nontrivial
element is usual conjugation, so Conjecture~\ref{cj:galois-re}
implies that there is some $m \in \Z/g\Z$ such that
$z_{m \cdot \frac{N}{g}} = \ol{z_0}$.

If $N$ is odd, then $\gcd(N, d) = 1$, so we have $z_0 = \ol{z_0}$,
i.e., $z_0$ is rational. Consequently the entire cycle lies within
$\Q$, so we can reduce the problem of finding quadratic $N$-periodic
points to that of finding rational $N$-cycles, which is a much more
approachable problem (the problem of finding rational points on curves
is studied extensively in the literature, whereas that of finding
points within quadratic fields is relatively obscure).

In particular, for $N = 5$, Flynn, Poonen, and Schaefer \cite{MR1480542}
showed that there are no rational 5-cycles, so we may
conclude that there are no quadratic 5-periodic points either.

\begin{corollary}
  \label{cor:n=5-0galois}
  If Conjecture~\ref{cj:galois-re} holds, then there are no rational
  $c$ such that $\phi_c(z) = z^2 + c$ has a 5-periodic point in $\Qq$.
\end{corollary}

In fact, in Section~\ref{sec:n=5} we will give an independent proof of
the finiteness of the total number of $c$'s that admit such 5-cycles,
and again conjecture that there are none based on empirical
observations.

On the other hand, If $N$ is even, then $\gcd(N, d) = 2$, and we have
either $z_0 = \ol{z_0}$, in which case the entire cycle lies within
$\Q$; or $z_0 = \ol{z_{\frac{N}{2}}}$, in which case
$z_0 + z_{\frac{N}{2}}$ is rational, and consequently the trace
$z_0 + z_1 + \cdots + z_{N-1} = \sum_{j = 0}^{\frac{N}{2} - 1} (z_j +
z_{j + \frac{N}{2}})$
is rational. Either case, the point on $C_0(N)$ that corresponds to
$c$ and the orbit $(z_0, \dots, z_{N-1})$ is a rational point, so the
problem of finding quadratic periodic points of exact period $N$ is
reduced to that of studying rational points on $C_0(N)$. This is again
considerably easier.

In particular, for $N = 6$, since the rational points on $C_0(6)$ are
already fully understood thanks to Stoll's work \cite{MR2465796}
(conditional on the weak Birch and Swinnerton-Dyer conjecture on the
Jacobian of $C_0(6)$), we can show by exhaustion that the only
quadratic 6-cycle is defined over $\Q(\sqrt{33})$, with
$c = -\frac{71}{48}$ and
\begin{equation}
  \label{eq:6-cycle}
  z_0 = -1 + \frac{\sqrt{33}}{12},\,
  z_1 = -\frac{1}{4} - \frac{\sqrt{33}}{6},\,
  z_2 = -\frac{1}{2} + \frac{\sqrt{33}}{12},\,
  z_3 = \ol{z_0},\,
  z_4 = \ol{z_1},\,
  z_5 = \ol{z_2}.
\end{equation}
Therefore, we have the following corollary.

\begin{corollary}
  \label{cor:6-cycle}
  Let $J$ be the Jacobian of $C_0(6)$. If
  \begin{enumerate}[(i)]
  \item The $L$-series $L(J, s)$ extends to an entire function and
    satisfies the standard functional equation;
  \item The weak Birch and Swinnerton-Dyer conjecture is valid for
    $J$; and
  \item Conjecture~\ref{cj:galois-re} holds,
  \end{enumerate}
  then the only rational $c$ such that $\phi_c(z) = z^2 + c$ has a
  6-periodic point in $\Qq$ is $-\frac{71}{48}$, and its corresponding
  periodic points are $z_0, \dots, z_5$ as defined in
  (\ref{eq:6-cycle}).
\end{corollary}

In summary, if Conjecture~\ref{cj:galois-re} is confirmed, then it
reduces all cases of finding quadratic $N$-periodic points to finding
rational points on $C_0(N)$, which is significantly easier than
finding points in quadratic extensions. In particular, for $N$ small
where we have a good understanding of $C_0(N)$, this leads to very
precise results.

\section{Another approach to the period 5 case}
\label{sec:n=5}

In the previous section, we argued that if we assume
Conjecture~\ref{cj:galois-re} (which is a sweeping conjecture that
applies to all values of $N$), then there are no quadratic periodic
points of period 5 (Corollary~\ref{cor:n=5-0galois}). In this section,
we use an entirely different approach in pursuit of proving
Conjecture~\ref{cj:n=5-zero} without the assumption of
Conjecture~\ref{cj:galois-re} as done in Corollary~\ref{cor:n=5-0galois}.
We work on a curve $C_P$ characterizing 5-cycles in $Qq$ using
information from $C_0(5)$ and search for the number of quadratic
polynomials $\phi_c(z) = z^2 + c$ with quadratic 5-periodic points.
We provide evidence that we find all such points and ultimately
suggest an application of the method of Chabauty and Coleman for
further progress.

Flynn, Poonen, and Schaefer showed in \cite{MR1480542} that $C_0(5)$,
which has genus 2, is birationally equivalent to the hyperelliptic
curve
\begin{equation}
  \label{eq:c0(5)}
  y^2 = f(x) = x^6 + 8x^5 + 22x^4 + 22x^3 + 5x^2 + 6x + 1,
\end{equation}
where the original $c$ is given in terms of $x$ and $y$ by
\begin{equation}
  \label{eq:c-in-xy}
  c = \frac{g(x)}{2(P_0(x) - P_1(x) y)}
  = \frac{P_0(x) + P_1(x) y}{h(x)},
\end{equation}
where $g, h, P_0, P_1 \in \Z[x]$ are
\begin{subequations}
  \label{eq:poly-defs}
  \begin{align}
    g(x) & = 8x^6 + 74x^5 + 271x^4 + 452x^3 + 325x^2 + 110x + 64,\\
    h(x) & = 8x^2(x+3)^2,\\
    P_0(x) & = - x^6 - 10x^5 - 46x^4 - 104x^3 - 95x^2 - 24x - 9,\\
    P_1(x) & = x^3 + 6x^2 + 3x - 9.
  \end{align}
\end{subequations}

We prove the following well-known result with the above model by
reducing quadratic points to rational points on a new curve.

Eliminating $y$ from (\ref{eq:c0(5)}) and (\ref{eq:c-in-xy}), we get
\[
(c \cdot h(x) - P_0(x))^2 = P_1(x)^2 y^2 = P_1(x)^2 f(x),
\]
i.e., the variable $x$ satisfies the polynomial equation
\[
[(c \cdot h - P_0)^2 - P_1^2 f](x) = 0.
\]
We are only looking for $x \in \Qq$,\footnote{%
  Recall that we are seeking $z \in \Qq$ and $c \in \Q$. On
  $C_0(5) = C_1(5) / \tup{\sigma}$, whose points correspond to pairs
  $(\O, \phi_c)$, each orbit $\O$ is represented by its trace
  $\tau = z + \phi_c(z) + \cdots + \phi_c^4(z)$, which is in the
  same quadratic extension as $z$. The series of change of
  coordinates from $(\tau, c)$ to $(x, y)$ involves only arithmetic
  operations, so the resulting $x$ and $y$ are still in the same
  quadratic extension as $\tau$. In particular, we have $x \in
  \Qq$. See \cite{MR1480542} for details.}
so $x$ satisfies a quadratic equation
\[
x^2 + a x + b = 0.
\]
for some rational coefficients $a$ and $b$. Having obtained two
polynomial equations in $x$, we may divide
$(c \cdot h - P_0)^2 - P_1^2 f$ by $x^2 + ax + b$ using long
division, and the remainder must be zero. We may compute that the
remainder is
\[
\l_1(a, b, c)\, x + \l_0(a, b, c),
\]
where $\l_1$ and $\l_0$ are polynomials in $a$, $b$ and $c$ with
integer coefficients, given by\footnote{%
  This and all subsequent computations in this proof can be found in
  our source code repository \cite{src} as
  \texttt{computational/mma/abc.m}.}
\[
\begin{aligned}
  \l_1&(a, b, c) = 16 (a - 3) (2b - a(a-3)) c^2 -
  4(a^{5} - 10 a^{4} - 4 a^{3} b + 46 a^{3} + 30 a^{2} b + \\
  & 3 a b^{2} - 104 a^{2} - 92 a b - 10 b^{2} + 95 a + 104 b - 24) c -
  (8 a^{5} - 74 a^{4} - 32 a^{3} b + \\
  & 271 a^{3} + 222 a^{2} b + 24 a b^{2} - 452 a^{2} - 542 a b - 74
  b^{2} + 325 a + 452 b - 110),
\end{aligned}
\]
and
\[
\begin{aligned}
  \l_0&(a, b, c) = 16 b (b - (a-3)^2) c^2 -
  4(a^{4} b - 10 a^{3} b - 3 a^{2} b^{2} + 46 a^{2} b + \\
  & 20 a b^{2} + b^{3} - 104 a b - 46 b^{2} + 95 b - 9) c -
  (8 a^{4} b - 74 a^{3} b - 24 a^{2} b^{2} + \\
  & 271 a^{2} b + 148 a b^{2} + 8 b^{3} - 452 a b - 271 b^{2} + 325 b
  - 64).
\end{aligned}
\]
Note that $a$, $b$ and $c$ are rational numbers, so $\l_1(a, b, c)$
and $\l_0(a, b, c)$ are both rational-valued, and hence from
\[
\l_1(a, b, c)\, x + \l_0(a, b, c) = 0
\]
we conclude that either $x$ is rational, or both $\l_1$ and $\l_0$
are zero. Thanks to \cite{MR1480542}, we already fully understand
the case where $x$ is rational (in which case there are no
corresponding periodic points in $\Qq$---the corresponding points
are either at infinity or are quintic over $\Q$). Therefore, we only
consider $x \in \Qq \setminus \Q$, in which case
\[
\l_1(a, b, c) = \l_0(a, b, c) = 0.
\]
Observe that $c$ is a common root to $\l_1$ and $\l_0$, so the
resultant of $\l_1$ and $\l_0$ with respect to $c$ is
zero. Depending on the degrees of $\l_1$ and $\l_0$ in $c$, we have
three cases.

\begin{case}
  The leading coefficient of $\l_1(a, b, c)$ (considered as a single
  variable polynomial in $c$) vanishes. This happens when
  \[
  16(a - 3)(2b - a(a-3)) = 0,
  \]
  i.e., either $a = 3$, or $b = a(a-3)/2$, or both.

  If $a = 3$, substituting $a = 3$ into $\l_1 = \l_2 = 0$ we get two
  polynomial equations in $b$ and $c$. Taking the resultant of the
  two with respect to $c$, we get
  \[
  - 96b^7 - 1264b^5 - 4256b^5 + 32b^4 + 13248b^3 + 37632b^2 + 92160b
  = 0,
  \]
  which has only one rational root $b = 0$. However,
  $\l_1(3, 0, c) \equiv -64 \ne 0$, a contradiction.

  If $b = a(a-3)/2$, substituting this into $\l_1 = \l_2 = 0$ we get
  two polynomials equations in $a$ and $c$. We can again reduce them
  to a single variable polynomial equation by taking the resultant,
  and easily derive a contradiction through exhaustion.
\end{case}

\begin{case}
  The leading coefficient of $\l_0(a, b, c)$ (considered as a single
  variable polynomial in $c$) vanishes. This happens when
  \[
  16b(b - (a-3)^2) = 0,
  \]
  i.e., either $b = 0$ or $b = (a - 3)^2$. Similar to Case~1, it is
  again a finite calculation to show that no rational values $a$,
  $b$, and $c$ work in this case.
\end{case}

\begin{case}
  Both of the leading coefficients of $\l_1(a, b, c)$ and
  $\l_0(a, b, c)$ (considered as single variable polynomials in $c$)
  are non-vanishing, i.e., both $\l_1(a, b, c)$ and $\l_0(a, b, c)$
  are quadratic in $c$. In this case, we compute the resultant of
  $\l_1$ and $\l_0$ with respect to $c$ directly (using
  Mathematica), which turns out to be a polynomial in $\Z[a,
  b]$. Denote this polynomial by $P(a, b)$.\footnote{%
    The polynomial $P(a, b)$ is very complicated: its degree in $a$
    is 8, and its degree in $b$ is 9.}  Our problem reduces to
  finding rational points $(a, b)$ on the curve $C_P$ defined by
  $P(a, b)$.
\end{case}

\begin{remark}
	A computation in Sage shows that the normalization of the curve
	$C_P$ has genus 11. Therefore, the number of rational points on
	this curve is finite by Faltings's Theorem on the Mordell
	Conjecture. For each rational point $(a, b)$, the
  rational value $c$ satisfies the polynomial equations
  $\l_1(a, b, c) = \l_0(a, b, c) = 0$, so the total number of $c$ is
  also finite. Hence we have given another demonstration of
	Propostion~\ref{prop:n=5-finite}.
\end{remark}

Since the finiteness of the number of rational $c$ such that $\phi_c$
has quadratic 5-periodic points is already known, the natural next step
is to find the precise number of such $c \in \Q$. However, we had already
found by Corollary~\ref{cor:n=5-0galois} that there are no such $c$ if
Conjecture~\ref{cj:galois-re} holds. Furthermore, a computational
search for such $c$ was unfruitful and provides further support that
there are no rational $c$ such that $\phi_c$ has quadratic 5-periodic
points. Here, we restate the corollary as a conjecture without the
conjectural hypothesis.

\begin{conjecture}[Restatement of Conjecture~\ref{cj:n=5-zero}]
  \label{cj:n=5-zero-re}
  There are no rational values $c$ such that $\phi_c(z) = z^2 + c$ has
  a 5-periodic point in $\Qq$.
\end{conjecture}

One promising approach to this conjecture is the study of the rational
points on the curve $C_P$ defined by $P(a, b)$, the resultant of
$\l_1$ and $\l_2$, as given above.  A full understanding of the rational
points of this curve will give complete information of the possible
values of $c$.

Computationally, we found 5 affine rational points $(3, 0)$, $(0, 0)$,
$(4, \frac{1}{3})$, $(1, \frac{8}{3})$ and $(6, 9)$ on $C_P$ with
small heights.\footnote{%
  Apart from the five known affine rational points, there are also
  three rational points at infinity. Five of these eight known
  projective rational points turn out to be regular, and the remaining
  three turn out to be nodes. With multiplicities counted, these
  amount to 11 known rational points in total on $C_P$.
}
It appears that these 5 affine rational points might be the only ones
on the curve and they do not correspond to $\phi_c$ with 5-periodic
points. Thus, proving that these are the only affine rational points
on $C_P$ will prove Conjecture~\ref{cj:n=5-zero-re}.

An application of Chabauty and Coleman's method \cite{MR808103} to
bound the number of rational points on $C_P$ may be fruitful, but the
bound that can be obtained from this method is estimated to be at
least 50. This bound would be too large for demonstrating the
nonexistence of rational points on $C_P$ outside of the 11 rational
points that we have already found. A clever refinement, such as the
technique used in \cite{MR1480542}, would be needed for any progress.

\section{Acknowledgements}
We are grateful to Niccol\`{o} Ronchetti for introducing us to this
field of study, providing incredibly helpful guidance, and being an
excellent project mentor. We thank Professor Brian Conrad for key
insights that led to the proof of Proposition~\ref{prop:n=5-finite}. We
also thank Professor Ben Hutz and Professor Patrick Ingram for their
remarks with regards to current results and progress in this field. We
are also appreciative of the Stanford Undergraduate Research Institute in
Mathematics (SURIM) for arranging our project and providing a great
environment for mathematical learning and collaboration. Finally, we
gratefully acknowledge that our research was financially supported by
research stipends from the Office of the Vice Provost for
Undergraduate Education (VPUE) of Stanford University. We deeply
appreciate all of the support that has made our work possible.

\nocite{*}
\bibliography{bibliography}{}

\begin{thebibliography}{10}

\bibitem{MR808103}
Robert~F. Coleman.
\newblock Effective {C}habauty.
\newblock {\em Duke Math. J.}, 52(3):765--770, 1985.

\bibitem{1309.6401}
John~R. Doyle, Xander Faber, and David Krumm.
\newblock Preperiodic points for quadratic polynomials over quadratic fields,
  2013.

\bibitem{MR1480542}
E.~V. Flynn, Bjorn Poonen, and Edward~F. Schaefer.
\newblock Cycles of quadratic polynomials and rational points on a genus-{$2$}
  curve.
\newblock {\em Duke Math. J.}, 90(3):435--463, 1997.

\bibitem{Hart2010}
W.~B. Hart.
\newblock Fast library for number theory: An introduction.
\newblock In {\em Proceedings of the Third International Congress on
  Mathematical Software}, ICMS'10, pages 88--91, Berlin, Heidelberg, 2010.
  Springer-Verlag.
\newblock \url{http://flintlib.org}.

\bibitem{MR3065461}
Benjamin Hutz and Patrick Ingram.
\newblock On {P}oonen's conjecture concerning rational preperiodic points of
  quadratic maps.
\newblock {\em Rocky Mountain J. Math.}, 43(1):193--204, 2013.

\bibitem{MR1199627}
Patrick Morton.
\newblock Arithmetic properties of periodic points of quadratic maps.
\newblock {\em Acta Arith.}, 62(4):343--372, 1992.

\bibitem{MR1665198}
Patrick Morton.
\newblock Arithmetic properties of periodic points of quadratic maps. {II}.
\newblock {\em Acta Arith.}, 87(2):89--102, 1998.

\bibitem{MortonPatel}
Patrick Morton and Pratiksha Patel.
\newblock The galois theory of periodic points of polynomial maps.
\newblock In {\em Proc. London Math. Soc. 68}, pages 225--263, 1994.

\bibitem{MR1264933}
Patrick Morton and Joseph~H. Silverman.
\newblock Rational periodic points of rational functions.
\newblock {\em Internat. Math. Res. Notices}, 1994(2):97--110, 1994.

\bibitem{MR0034607}
D.~G. Northcott.
\newblock Periodic points on an algebraic variety.
\newblock {\em Ann. of Math. (2)}, 51:167--177, 1950.

\bibitem{MR2982105}
Chatchawan Panraksa.
\newblock {\em Arithmetic dynamics of quadratic polynomials and dynamical
  units}.
\newblock PhD thesis, University of Maryland, College Park, 2011.

\bibitem{MR1956273}
Bjorn Poonen.
\newblock Computing rational points on curves.
\newblock In {\em Number theory for the millennium, {III} ({U}rbana, {IL},
  2000)}, pages 149--172. A K Peters, Natick, MA, 2002.

\bibitem{MR2316407}
Joseph~H. Silverman.
\newblock {\em The arithmetic of dynamical systems}, volume 241 of {\em
  Graduate Texts in Mathematics}.
\newblock Springer, New York, 2007.

\bibitem{MR2465796}
Michael Stoll.
\newblock Rational 6-cycles under iteration of quadratic polynomials.
\newblock {\em LMS J. Comput. Math.}, 11:367--380, 2008.

\bibitem{MR2780629}
Michael Stoll.
\newblock Rational points on curves.
\newblock {\em J. Th\'eor. Nombres Bordeaux}, 23(1):257--277, 2011.

\bibitem{src}
Zhiming Wang.
\newblock Computational programs used in this paper.
\newblock \url{https://github.com/surim14/arith-dyn}.

\end{thebibliography}
\bibliographystyle{plain}

\end{document}